\documentclass{article}

\usepackage[english]{babel}

\usepackage[letterpaper,top=2cm,bottom=2cm,left=3cm,right=3cm,marginparwidth=1.75cm]{geometry}

\usepackage{amsmath}
\usepackage{amsthm}
\usepackage{amsfonts,amssymb}
\usepackage{graphicx}
\usepackage{hyperref}
\hypersetup{colorlinks=true, allcolors=blue}
\usepackage{physics} 
\usepackage{siunitx} 
\usepackage{graphicx,pstricks,listings,stackengine}
\usepackage{enumerate} 
\usepackage{authblk}
\newtheorem{defn}{Definition}
\newtheorem{thm}{Theorem}
\newtheorem{prop}{Proposition}
\newtheorem{lem}{Lemma}

\title{Finding Sums of Four Squares via Complex Continued Fractions}
\author{Zhaonan Wang}
\author{Yingpu Deng}
\affil{Key Laboratory of Mathematics Mechanization, NCMIS, Academy of Mathematics and Systems Science, Chinese Academy of Sciences, Beijing 100190, People's Republic of China\authorcr and\authorcr 
	School of Mathematical Sciences, University of Chinese Academy of Sciences,
	Beijing 100049, People’s Republic of China
	\authorcr
	znwang@amss.ac.cn, dengyp@amss.ac.cn}
\date{}
\begin{document}
	\bibliographystyle{plain}
	\maketitle
	\begin{abstract}
		The problem of representing a given positive integer as a sum of four squares of integers has been widely concerned for a long time, and for a given positive odd $n$ one can find a representation by doing arithmetic in a maximal order of quaternion algebra once a pair of (positive) integers $x,y$ with $x^2+y^2\equiv-1\mod n$ is given. In this paper, we introduce a new method to find a representation of odd integer $w$ given $x,y$ satisfying the above requirement. This method can avoid the complicated non-commutative structure in quaternion algebra, which is similar to the one we use to obtain a representation of a prime $p\equiv1\mod4$ as sum of two squares by doing continued fraction expansions, except that here we will expand complex number using Hurwitz algorithm. 
	\end{abstract}
	\section{Introduction}
	In 1770 Lagrange proved in \cite{Lagr} that all positive integers can be written as a sum of four squares. In 1986, three randomized algorithms were presented by Rabin and Shallit\cite{Rab} to obtain one representation for any given (positive) integer $n$ assuming the truth of the Extended Riemann Hypothesis, and one of them used a maximal order called Hurwitz order $\mathbf{H}=\{\frac{h_1+h_2i+h_3j+h_4k}{2}\ |\ all\ h_i\in\mathbb{Z}, \ h_1\equiv h_2\equiv h_3\equiv h_4\mod2\}$, which is contained in the restriction of Hamiltonians from $\mathbb{R}$ to $\mathbb{Q}$, namely the quaternion algebra $(\frac{-1,-1}{\mathbb{Q}})$, where $i,\ j$ and $k$ are the coordinates satisfying $i^2=j^2=k^2=-1$ and $ij=k$, $jk=i$, $ki=j$. Once a solution $x^2+y^2\equiv-1\mod n$ has been found, one can quickly write $n$ as a sum of four squares by computing the greatest right common divisor of $x+yi+j$ and $n$ in $\mathbf{H}$ \cite{pollack}.\\
	However, this is much less satisfactory than the method of the sum-of-two-squares problem, since the intricate structure of $(\frac{-1,-1}{\mathbb{Q}})$ will bring much trouble to the calculation in the algebra due to its non-commutative nature. Therefore, we aim to present an algorithm that can work in a commutative ring, i.e., the Gaussian integer ring $\mathbb{Z}[i]$ and avoid doing arithmetic in the Hurwitz order.\\
	To help understand the basic idea of this paper, we illustrate here the method Hermite raised in 1848 \cite{Herm} for representing a given prime $p\equiv1\mod4$ as a sum of two squares (One can also see \cite{Bri}):
	\begin{enumerate}
		\item Find $x_0$ with $0<x_0<\frac{p}{2}$ such that $x_0^2\equiv-1\mod p$.
		\item Expand $\frac{x_0}{p}$ into a simple continued fraction expansion till the denominator of the convergents $\frac{P_n}{Q_n}$ satisfies $Q_n<\sqrt{p}<Q_{n+1}$. Then we have
	\end{enumerate}
	$$p=(x_0Q_n-p P_n)^2+Q_n^2$$
	This method has enlightened us to come up with the idea of using ways of continued fraction expansions to consider the sum-of-four-squares problem. However, here we are concerned with continued fraction expansion of complex numbers instead of the classical  continued fractions.\\
	The first attempt about complex continued fractions was made by Hurwitz \cite{Hur} in 1887, where he developed an expanding algorithm of choosing the nearest Gaussian integers each time and proved some similar properties as in the classical continued fractions. In \cite{Doug}, Hensley described some further properties of the Hurwitz continued fraction expansions, like the growth of the absolute value of denominators of the approximation and distribution of remainders, etc.. Although they focused on infinite expansions rather than expansions of rational complex numbers that are finite, most of the theorems and properties they developed still hold for finite cases. \\
	The paper is organized as follows. In Sec.\ref{sec2} we review some properties about Hurwitz continued fractions and lattices that shall be used later. We prove our main theorem in Sec.\ref{sec3} and finally propose an algorithm to describe the whole process along with an example to find a representation as a sum of four squares for a large odd prime $\equiv3\mod4$ in Sec.\ref{sec4}.
	
	\section{Preliminaries of complex continued fractions and lattices}\label{sec2}
	First we present some basic facts about complex continued fractions and Hurwitz's algorithm about complex expansions. The reader can refer to Section 5.2 in \cite{Doug} for more details, and some of the notations in this paper are adopted from \cite{Dani}. \\
	Let $\mathfrak{G}$ denote the Gaussian integer ring $\mathbb{Z}[i]$, and $\{a_n\}$ be a sequence in $\mathfrak{G}$, which can be finite or infinite. Define the $\mathcal{Q}$-pair $\{P_n\}$ and $\{Q_n\}$ of sequences associated to $\{a_n\}$ recursively as follows:
	$$P_{-1}=1,\ P_0=a_0,\ P_{n+1}=a_{n+1}P_n+P_{n-1}\ (n\geq0)$$
	$$Q_{-1}=0,\ Q_0=1,\ Q_{n+1}=a_{n+1}Q_n+Q_{n-1}\ (n\geq0)$$
	It's easy to verify that $P_n Q_{n-1}-Q_n P_{n-1}=(-1)^{n-1}$ for all $n\geq0$. and $\frac{P_n}{Q_n}=a_0+\frac{1}{a_1+\frac{1}{\cdots+\frac{1}{a_n}}}$ can be regarded as the $n$-th convergent defined by the sequence $\{a_n\}$, if the sequence $a_n$ is infinite and $\frac{P_n}{Q_n}$ converge when $n\rightarrow{\infty}$, we may say that the complex number $z:=a_0+\frac{1}{a_1+\frac{1}{a_2+\cdots}}$ owns a continued fraction expansion $[a_0;a_1,a_2,\cdots]$ as in the classical  continued fraction case.\\
	However, defining a continued fraction expansion for a given complex $z$ requires much more notations and definitions. Since we do not focus on those details for continued fraction expansion algorithm, here we only introduce the Hurwitz algorithm which is involved in this paper. We denote by $[z]$ the Gaussian integer nearest $z$, i.e., rounding down both real and imaginary parts of $z $. The Hurwitz algorithm is more likely to be an improvement of the classical centered continued fraction algorithm for real numbers, which proceeds by defining the two sequences recursively as follows (given $z=z_0$):
	$$a_n=[z_n]$$
	$$z_{n+1}=(z_n-a_n )^{-1}$$
	We call $\{z_n\}$ the iteration sequence and $\{a_n \}$ the partial quotients of $z $. Again we can define the $\mathcal{Q}$-pair $\{P_n\}$ and $\{Q_n\}$ as above, and they still satisfy the recursive equalities.\\
	Since $z_n-a_n$ lies in $\Phi:=\{x+yi\ |-\frac{1}{2}\leq x,y\leq\frac{1}{2}\}$, which can be seen from the definition of $a_n $, we have $z_{n}\in\Phi^{-1}=\{x+yi\ |\ (|x|-1)^2+y^2\geq1,\ x^2+(|y|-1)^2\geq1\}$ for $n\geq1$, then $|\Re(z_n)|\geq1$, $|\Im(z_n)|\geq1$, and $a_n\in\mathfrak{G}\backslash\{0, \pm 1, \pm i\}$ for all $n\geq1$.
	\begin{prop}\label{prop1}
		Let $z\in\mathbb{C}$, $\{z_n\}$ and $\{a_n\}$ be the iteration sequence and partial quotients of $z$ under Hurwitz algorithm, respectively. Let $\{P_n\}$ and $\{Q_n\}$ be the $\mathcal{Q}$-pair associated to $\{a_n \}$. Then we have $Q_{n} z-P_{n}=(-1)^{n}\left(z_{1} \cdots z_{n+1}\right)^{-1}$ for all legal $n$, and $z_{n+1}=-\frac{Q_{n-1} z-P_{n-1}}{Q_{n} z-P_{n}}$.
	\end{prop}
	Here by the word $legal$ we mean $n$ can be arbitrarily chosen if $z$ admits an infinite continued fraction expansion, or $n\leq m$ if $z=[a_0;a_1,\cdots,a_m]$. 
	One can prove this proposition by a simple induction on $n$ or just see Proposition 3.3 in \cite{Dani}.\\
	In this paper we always deal with the case where $z\in\mathbb{Q}[i]$, i.e., $z$ being a rational complex number, hence the algorithm will always terminate. If the expansion goes to depth $m$, then $z_m=0$ and $\frac{P_m}{Q_m}=z$. Notice that if $z\in\mathbb{Q}[i]$ and has a Hurwitz expansion $[a_0;a_1,\cdots,a_m]$, then for any $n\geq0$, we can obtain a Hurwitz expansion of $z_n=[a_{n};a_{n+1},\cdots,a_m]$.\\
	Now we illustrate the definition and some fundamental properties of lattices, which shall be used in the proof of the main theorem. One can check Chapter 1 in \cite{Carl} for more details.
	\begin{defn}
		$\boldsymbol{(Lattice)}$ Let $\{\mathbf{b_1},\mathbf{b_2},\cdots,\mathbf{b_k}\}$ be $k$ linearly independent vectors in $\mathbb{R}^{n}$,  $\mathcal{L}$ is a lattice of rank $k$ with basis $\{\mathbf{b_1},\mathbf{b_2},\cdots,\mathbf{b_k}\}$ if $\mathcal{L}$ is the set of all points $\sum\limits_{i=1 }^k x_i \mathbf{b_i}$ with integral $x_1,\cdots,x_k$.
	\end{defn}
	Sometimes we simply use the matrix $(\mathbf{b_1},\cdots,\,\mathbf{b_k})$ to denote the lattice generated by vectors $\mathbf{b_1},\mathbf{b_2},\cdots,\mathbf{b_k}$. In this paper we only consider full rank lattice, i.e., the case where $k=n$.
	\begin{defn}
		Given a full-rank lattice $\mathcal{L}=(\mathbf{b_1},\cdots,\mathbf{b_n})$ in $\mathbb{R}^n$,  the determinant of $\mathcal{L}$ is defined as below:
		$$\det(\mathcal{L}):=|\mathbb{Z}^n/\mathcal{L}|=|\det(\mathbf{b_1},\cdots,\mathbf{b_n})|$$
		and one can verify that this determinant remains unchanged for any basis of $\mathcal{L}$.  
	\end{defn}
	\begin{prop}\label{prop2}
		$\boldsymbol{(Minkowski\ Theorem\ 1)}$ Given a full-rank lattice $\mathcal{L}$ in $\mathbb{R}^n$, any convex centrally symmetric body $S$ of volume $>2^n\det(\mathcal{L})$  contains a nonzero lattice point in $\mathcal{L}$.
	\end{prop}
	\section{Representing $w$ as a sum of four squares}\label{sec3}
	Next we present the main theorem of this paper to represent integer $w$ as a sum of four squares. We may assume $w$ is odd. To see this, consider arbitrary $w'=2^ew$ with $w$ odd. If we have $a^2+b^2+c^2+d^2=w$, then we can derive $a',b',c',d'$ with $a'^2+b'^2+c'^2+d'^2=w'$ defined by
	$$(1+i)^e(a+bi+cj+dk)=a'+b'i+c'j+d'k$$
	which can be computed quickly in the quaternion algebra.\\
	Our result more like an alternative to the method of computing $\mathrm{g.r.c.d.}(w,x+yi+j)$ in $(\frac{-1,-1}{\mathbb{Q}})$. The theorem holds for all odd positive integers that are not perfect squares, but we only have deterministic polynomial algorithm for finding solutions $x,y$ to $x^2+y^2\equiv-1\mod p$ for prime $p\equiv3\mod4$, one can check \cite{Sch} for details. As for general $w$, Rabin and Shallit described a randomized algorithm to find $x,y$, see Theorem 3.1 in \cite{Rab}. Also, Pollard and Schnorr\cite{Poll} have given a more general algorithm for solving $x^2+Dy^2\equiv k\mod w$ with which runs quickly in random polynomial time under the assumption of General Riemann Hypothesis (GRH).\\
	We define $|z|:=\sqrt{\Im^2(z)+\Re^2(z)}$ to be the norm of $z$ if $z\in\mathbb{C}$, and for a vector $\mathbf{b}=(b_1,\cdots,b_n)$ $||\mathbf{b}||:=\sqrt{b_1^2+\cdots+b_n^2}$ is the length of $\mathbf{b}$. Our main theorem is stated as follows:
	\begin{thm}\label{thm1}
		Given an odd integer $w$ (not a square) and integers $x,\ y$ $(0\leq x,y<\frac{w}{2})$ s.t. $x^2+y^2\equiv-1\mod{w}$, $z:=\frac{x+yi}{w}$ admits a Hurwitz continued fraction (HCF) expansion $[a_0,\cdots,a_m]=[0;a_1,\cdots,a_m]$ with $\mathcal{Q}$-pairs $(P_k,Q_k)$. One can find unique $n<m$ such that $| Q_n|\leq\sqrt{w}<| Q_{n+1}|$, if $|Q_n |\neq\sqrt{w}$, then $w=|(x+yi)\cdot Q_n-w\cdot P_n)|^2+| Q_n|^2$, and one can obtain a representation of $w$ as a sum of four squares.
	\end{thm}
	\begin{proof}
		We may assume that all Gaussian integers appearing in the context below do not have norm $\sqrt{w}$, otherwise we can obtain a representation as sum of two squares of $w$ and the problem is solved.\\ 
		Here are two lemmas used to prove the theorem. One can find their proof in \cite{Hur} and \cite{Lak}. Though both of them originally discussed the infinite expansion case, they are still available for rational complex numbers with exactly the same proof.
		\begin{lem}\label{lem1}
			$\boldsymbol{(Hurwitz, 1887)}$ For any complex number $z$ with HCF expansion $[a_0;a_1,\cdots,a_m,\cdots]$ and $Q$ pairs $(P_k,Q_k)$, there is $$1=| Q_0|<| Q_1|<\cdots<| Q_m|$$ for all legal $m$.
		\end{lem}
		\begin{lem}\label{lem2}
			$\boldsymbol{(First\ Lakein\ Theorem)}$ For any complex number $z$ with HCF expansion $[a_0;a_1,\cdots,a_m,\cdots]$ and any legal $k$, $| z-\frac{P_k}{Q_k}|\leq\frac{1}{| Q_k|^2}$. 
		\end{lem}
		From Lemma \ref{lem1} we know that in our case there exists a unique $n<m$ such that $| Q_n|<\sqrt{w}<| Q_{n+1}|$ holds, and by easy calculation one can find $w\ |\ |((x+yi)\cdot Q_k-w\cdot P_k)|^2+| Q_k|^2$ for arbitrary $k$ using the condition $w\ |\ x^2+y^2+1$. Now we denote $S_k:=(x+yi)\cdot Q_k-w\cdot P_k$ for all $-1\leq k\leq m$, and show $|S_n|^2+|Q_n|^2$ can only be $w$ or $2w$ with the $n$ chosen as in the theorem.\\
		If $|S_n|<\sqrt{w}$, we can easily know that $|S_n|^2+|Q_n|^2=w$ since $|Q_n|<\sqrt{w}$ and $w\ |\ |S_n|^2+|Q_n|^2< 2 w$. Therefore we just obtain a representation of $w$ as sum of four squares. From now on we suppose that $|S_n|>\sqrt{w}$.\\
		From Proposition \ref{prop1} we have $z_{k+1}=-\frac{Q_{k-1}\cdot\frac{x+yi}{w}-P_{k-1}}{Q_k\cdot\frac{x+yi}{w}-P_k}=-\frac{S_{k-1}}{S_k}$ for all $k$. By rules of Hurwitz algorithm we know $| z_{k+1}|\geq\sqrt{2}$ and thus $| S_{k-1}|\geq\sqrt{2}\ | S_{k}|$.\\
		For any $k$, $S_k Q_{k+1}-S_{k+1}Q_k=((x+yi) Q_k-w P_k)Q_{k+1}-((x+yi)  Q_{k+1}-w P_{k+1})Q_k=(-1)^kw$. Then take $k=n$ and we obtain $S_nQ_{n+1}-S_{n+1}Q_n=(-1)^n w$. Firstly consider the case where $n$ is even, hence now we have $S_n Q_{n+1}-S_{n+1}Q_n=w$. Next we list some facts that can be obtained from the previous content. From Lemma \ref{lem2} we have $| Q_n S_{n}|=|Q_n((x+yi)\cdot Q_n-w\cdot P_n)|=|w\cdot Q_n^2(\frac{x+yi}{w}-\frac{P_n}{Q_n})|\leq w$, $| S_{n+1}|=|w\cdot Q_{n+1}(z-\frac{P_{n+1}}{Q_{n+1}})|\leq\frac{w}{| Q_{n+1}|}<\sqrt{w}$, and recall that we have $| Q_n|<\sqrt{w}$, $| S_n|>\sqrt{w}$, hence $| Q_{n+1}S_n|>w>| Q_nS_{n+1}|$. Also $|\frac{Q_{n+1}S_{n}}{Q_nS_{n+1}}|=|\frac{S_n}{S_{n+1}}|\cdot|\frac{Q_{n+1}}{Q_n}|>\sqrt{2}$, $| Q_{n+1}S_{n}|\cdot| Q_n S_{n+1}|=|(Q_n S_n)\cdot(Q_{n+1}S_{n+1})|\leq w^2$.\\
		Now we use the facts above to derive more restrictions for $Q_{n+1}S_n$ to satisfy. Assume $Q_nS_{n+1}=r+ti$, then $Q_{n+1}S_n=r+w+ti$. We have
		$$\left\{\begin{array}{l}
			r^2+t^2<w^2 \\
			(r+w)^2+t^2>w^2 \\
			(r+w)^2+t^2>2(r^2+t^2) \\
			((r+w)^2+t^2)\cdot(r^2+t^2)\leq w^4
		\end{array}\right.$$
		We consider the cases $r<0$ and $r\geq0$ separately to obtain some upper bounds for $|Q_{n+1}S_n|$.
		When $r<0$, $(r+w)^2+t^2<2 w^2+2 wr\leq2 w^2$, hence $| Q_{n+1}S_n|^2<2w^2$.
		When $r\geq0$, the second inequality in the array naturally holds, and the third one holds if the first holds. Hence we may assume $r=w\epsilon \cos\theta$, $t=w\epsilon \sin\theta$, where $\epsilon\in (0,1)$ from the first inequality and $\theta\in[-\frac{\pi}{2},\frac{\pi}{2}]$ from the assumption that $r\geq0$. Now we have $| Q_{n+1}S_n|^2=(r+w)^2+t^2=w^2\epsilon^2+w^2+2w^2\epsilon \cos\theta$, and from the last inequality we have $\epsilon^2(\epsilon^2+1+2\epsilon \cos\theta)\leq1$.\\
		Recall that our goal is to obtain the upper bound of $|Q_{n+1}S_n|$ (equivalently the upper bound of $w^2\epsilon^2+w^2+2w^2\epsilon \cos\theta$), and the last inequality tells us $\cos\theta\leq\frac{1}{2\epsilon}(\frac{1}{\epsilon^2}-\epsilon^2-1)$. On one hand, $\cos\theta\geq0$, thus we must always have $\frac{1}{2\epsilon}(\frac{1}{\epsilon^2}-\epsilon^2-1)\geq0$, after solving this we have $\epsilon^2\leq\frac{\sqrt{5}-1}{2}$. On the other hand $\cos\theta\leq1$, so if we require $\frac{1}{2\epsilon}(\frac{1}{\epsilon^2}-\epsilon^2-1)\geq1$, the inequality $\cos\theta\leq\frac{1}{2\epsilon}(\frac{1}{\epsilon^2}-\epsilon^2-1)$ would always hold. Easy to see $\frac{1}{2\epsilon}(\frac{1}{\epsilon^2}-\epsilon^2-1)$ decreases when $\epsilon$ increases, and one can find the root of $\frac{1}{2\epsilon}(\frac{1}{\epsilon^2}-\epsilon^2-1)=1$, which is $\epsilon=\frac{\sqrt{5}-1}{2}$. Therefore, when $\epsilon\leq\frac{\sqrt{5}-1}{2}$, $\cos\theta$ can be chosen arbitrarily.\\ 
		Next we shall find the upper bounds separately for both $0<\epsilon\leq\frac{\sqrt{5}-1}{2}$ and $\frac{\sqrt{5}-1}{2}<\epsilon\leq\sqrt{\frac{\sqrt{5}-1}{2}}$. When $0<\epsilon\leq\frac{\sqrt{5}-1}{2}$, $\cos\theta$ can be chosen arbitrarily by the last paragraph, hence to make $w^2\epsilon^2+w^2+2w^2\epsilon \cos\theta$ possibly large, we take $\cos\theta=1$, $\epsilon=\frac{\sqrt{5}-1}{2}$. Therefore, $\mathrm{max}\{ w^2\epsilon^2+w^2+2w^2\epsilon \cos\theta\ |\ 0<\epsilon\leq\frac{\sqrt{5}-1}{2} \}=(\frac{\sqrt{5}+1}{2})^2 w^2$.\\
		When $\frac{\sqrt{5}-1}{2}<\epsilon\leq\sqrt{\frac{\sqrt{5}-1}{2}}$, recall that we have the inequality $\epsilon^2(\epsilon^2+1+2\epsilon \cos\theta)\leq1$, hence $\mathrm{max}\{ w^2\epsilon^2+w^2+2 w^2\epsilon \cos\theta\ |\ \frac{\sqrt{5}-1}{2}<\epsilon\leq(\frac{\sqrt{5}-1}{2})^{\frac{1}{2}} \}\leq\mathrm{max}\{ \frac{1}{\epsilon^2}w^2\ |\ \frac{\sqrt{5}-1}{2}<\epsilon\leq({\frac{\sqrt{5}-1}{2}})^{\frac{1}{2}} \}=(\frac{\sqrt{5}+1}{2})^2 w^2$.\\
		Therefore we always have $|Q_{n+1}S_n|\leq\frac{\sqrt{5}+1}{2}w$ when $n$ is even. As for the case where $n$ is odd, we have $S_nQ_{n+1}-S_{n+1}Q_n=-w$. The analysis above in $n$-is-even case still holds, and we may assume that $Q_nS_{n+1}=r+ti$, $Q_{n+1}S_n=r-w+ti$. Then we obtain four inequalities again:
		$$\left\{\begin{array}{l}
			r^2+t^2<w^2 \\
			(r-w)^2+t^2>w^2 \\
			(r-w)^2+t^2>2(r^2+t^2) \\
			((r-w)^2+t^2)\cdot(r^2+t^2)\leq w^4
		\end{array}\right.$$
		Still we consider $r>0$ and $r\leq0$ separately and the upper bounds for $|Q_{n+1}S_n|^2$ is exactly the same as in the $n$-is-even case. \\
		Finally we have $|Q_{n+1}S_n|\leq\frac{\sqrt{5}+1}{2}w$ in all cases.
		Since $| Q_{n+1}|>\sqrt{w}$, we have $| S_n|<\frac{\sqrt{5}+1}{2}\sqrt{w}$, and $| S_n|^2+| Q_n|^2<\frac{5+\sqrt{5}}{2}w<4 w$, hence can only be $2 w$ or $3 w$. \\
		Now we claim that $| S_n|^2+| Q_n|^2=3 w$ is impossible. Otherwise from Lemma \ref{lem2} we know that $| \frac{w}{Q_n}|^2+| Q_n|^2\geq3 w$, solving this inequality we have $| Q_n|^2\leq\frac{3-\sqrt{5}}{2}w$, hence we must have $| S_n|^2\geq\frac{3+\sqrt{5}}{2}w$ (i.e. $|S_n|\geq\frac{\sqrt{5}+1}{2}\sqrt{w}$) to make the equality $|S_n|^2+| Q_n|^2=3 w$ holds, contradicting to the condition $| S_n|<\frac{\sqrt{5}+1}{2}\sqrt{w}$ that we've obtained in the last paragraph. \\
		Finally we come to the conclusion that $| S_n|^2+| Q_n|^2$ can only be $2 w$ when $|S_n|>\sqrt{w}$. Notice that $S_n$ and $Q_{n+1}$, $S_{n+1}$ and $Q_n$ can be considered symmetrically somehow, since in both pairs the elements share the same properties. Therefore for the same reason as discussed above we have the equation $|S_{n+1}|^2+|Q_{n+1}|^2=2 w$ under the assumption that $|S_n|>\sqrt{w}$. Next we will prove that these two equations for $|S_n |$, $|Q_n|$, $|S_{n+1}|$ and $|Q_{n+1}|$ cannot hold at the same time. To do this we will introduce some more definitions and propositions.
		
		First consider a full-rank lattice $\mathcal{L}=(\mathbf{b_1},\mathbf{b_2},\mathbf{b_3},\mathbf{b_4})$ in $\mathbb{R}^4$:
		$$\left(\begin{array}{cccc}
			0 & 0 & 1 & 0 \\
			0 & 0 & 0 & 1 \\
			0 & -w & x & y \\
			-w & 0 & y & -x \\
		\end{array}\right)$$
		Easy to see $\det(\mathcal{L})=w^2$. For any $\mathbf{b}\in\mathcal{L}$, write $\mathbf{b}=\sum\limits_{i=1}^4\mu_i\mathbf{b_i}$, then $$||\mathbf{b}||^2=(-w\mu_1+y\mu_3-x\mu_4)^2+(-w\mu_2+x\mu_3-y\mu_4)^2+\mu_3^2+\mu_4^2$$
		Thus $||\mathbf{b}||^2\equiv(x^2+y^2+1)(\mu_3^2+\mu_4^2)\equiv0\mod w$, i.e., $w\ |\ ||\mathbf{b}||^2$ for any $\mathbf{b}\in\mathcal{L}$. \\
		Consider the ball centered at the origin in $\mathbb{R}^4$ with radius $\sqrt{2w-\epsilon}$ for some small $\epsilon>0$ and denote it by $B(\sqrt{2 w-\epsilon})$. Notice that when $\epsilon$ is sufficiently small,
		$${\rm Vol} (B(\sqrt{2w-\epsilon}))=\frac{\pi^2}{2}(2w-\epsilon)^2>2^4w^2=2^4\det(\mathcal{L})$$ 
		hence $B(\sqrt{2w-\epsilon})$ must contain a nonzero point $\mathbf{u}$ in $\mathcal{L}$ by Proposition \ref{prop2}. Write $\mathbf{u}=\sum\limits_{i=1}^4 x_i \mathbf{b_i}$, and $w\ |\ ||u||^2=(-x_1w+x_3y-x_4x)^2+(-x_2w+x_3x-x_4y)^2+x_3^2+x_4^2<2 w$, which implies that
		$$(-wx_1+yx_3-xx_4)^2+(-wx_2+xx_3-yx_4)^2+x_3^2+x_4^2=w$$
		Alternatively speaking we just find two Gaussian integers $x_1+x_2i$ and $x_3+x_4i$ such that
		\begin{equation}\label{eq1}
			|x_3+x_4i|^2+|(x_1+x_2i)w-(x_3+x_4i)(x+yi)|^2=w
		\end{equation}
		On the other hand if we take $(x_1+x_2i)w-(x_3+x_4i)(x+yi)=x_3'-x_4'i$, by direct calculation we have
		$$(x_3-x_4i)-(x_3'+x_4'i)(x+yi)=(x^2+y^2+1)(x_3-x_4i)+(x_1-x_2i)w$$
		which is divisible by $w$. Therefore there exists some $x_1'+x_2'i\in\mathbb{Z}[i]$ such that
		\begin{equation}\label{eq2}
			|x_3'+x_4'i|^2+|(x_3'+x_4'i)(x+yi)-(x_1'+x_2'i)w|^2=w
		\end{equation}
		where $x_3'+x_4'i=\overline{(x_1+x_2i)w-(x_3+x_4i)(x+yi)}$, $(x_3'+x_4'i)(x+yi)-(x_1'+x_2'i)w=\overline{x_3+x_4i}$\\
		In other words, (\ref{eq1}) and (\ref{eq2}) are essentially the same, being two interpretations of one equation.\\
		We may assume that $|x_3+x_4i |^2>\frac{w}{2}>|x_3'+x_4'i|^2$, i.e., 
		$$w>|x_3+x_4i|^2>\frac{w}{2}>|(x_1+x_2i)w-(x_3+x_4i)(x+yi)|^2$$
		( Here we still omit the case where some Gaussian integers have norm $\sqrt{w}$.)\\
		Now we introduce another lemma originally proved by Lakein.
		\begin{defn}
			$\boldsymbol{(Good\ Approximation)}$ Let $z\in\mathbb{C}$ be a complex number. A rational complex $\frac{p}{q}$ $(p,q\in\mathbb{Z}[i])$ is a good approximation to $z$ if for any $p',q'\in\mathbb{Z}[i]$ with $|q'|\leq|q|$, $|q'z-p'|\leq|qz-p|$.
		\end{defn}
		\begin{lem}\label{lem3}
			$\boldsymbol{(Second\ Lakein\ Theorem)}$ If $z\in\mathbb{C}$ admits a HCF expansion, then any HCF convergent $\frac{P_k}{Q_k}$ of $z$  is a good approximation to $z$.
		\end{lem}
		Still one can find proof in Lakein's work \cite{Lak2}. Alternatively one can check \cite{Rob} for all three lemmas' proof.\\
		Consider firstly the index $k_1$ such that $|Q_{k_1}|<|x_3+x_4i|\leq|Q_{k_1+1}|$. Note that such $k_1$ must exist since $|Q_{-1}|=0$, $|Q_m|=w$ and $|Q_k|$ monotonically increases as $k$ increases. Recall that now we have $|Q_n|<\sqrt{w}<|Q_{n+1}|$ and $|S_n|>\sqrt{w}>|S_{n+1}|$.\\
		Since $|x_3+x_4i|\leq|Q_{k_1+1}|$, take $q'=x_3+x_4i$, $p'=x_1+x_2i$, from Lemma \ref{lem3} we have $|S_{k_1+1}|\leq|q'(x+yi)-p'w|$, i.e., $|S_{k_1+1}|\leq|x_3'+x_4'i|<\sqrt{\frac{w}{2}}$. Therefore $k_1+1\geq n+1$ from the selection of $n$. However we also have $|Q_{k_1}|<|x_3+x_4i|<\sqrt{w}$, hence $k_1\leq n$. Combining the results we have $k_1=n$.\\
		Similarly we consider the index $k_2$ such that $|Q_{k_2}|<|x_3'+x_4'i|\leq |Q_{k_2+1}|$ and by the same discussion as above we obtain $k_2=n$. Recall that $|x_3'+x_4'i|<\sqrt{\frac{w}{2}}$, hence now we have
		\begin{equation*}
			|Q_n|<\sqrt{\frac{w}{2}}<\sqrt{w}<|Q_{n+1}|
		\end{equation*}
		\begin{equation*}
			|S_n|>\sqrt{w}>\sqrt{\frac{w}{2}}>|S_{n+1}|
		\end{equation*}
		Since $|S_n|^2+|Q_n|^2=|S_{n+1}|^2+|Q_{n+1}|^2=2 w$, there must be $|S_n|^2>\frac{3w}{2}$, $|Q_{n+1}|^2>\frac{3w}{2}$, thus
		\begin{equation}\label{eq3}
			|S_nQ_{n+1}|>\frac{3w}{2}
		\end{equation}
		\begin{equation}\label{eq4}
			|S_{n+1}Q_{n}|<\frac{w}{2}    
		\end{equation}
		
		However we already have
		\begin{equation}\label{eq5}
			S_nQ_{n+1}-S_{n+1}Q_n=(-1)^nw
		\end{equation}
		From the triangular inequality we know (\ref{eq3}), (\ref{eq4}), (\ref{eq5}) cannot hold at the same time. Thus the assumption $|S_n|>\sqrt{w}$ is not true and we complete the proof of the theorem.
	\end{proof}
	\section{Algorithms and Examples}\label{sec4} 
	Now we summarize the content in Section \ref{sec3} as an algorithm to obtain a representation as sum of four squares for odd $w$. \\
	Since for prime $p\equiv1\mod4$, there already exists algorithms running in polynomial time that can find $x,y$ such that $p=x^2+y^2$ (see \cite{Sch}), hence we mainly consider the odd $w$ that is not a prime in the form $4k+1$.
	\begin{center}
		\begin{tabular}{ l }
			\hline
			$\mathbf{Algorithm\ 1}$: Finding a representation of odd $w$ as a sum of four squares \\[0.4ex]
			\hline
			$\mathbf{Input:}$ An odd positive integer $w$ not a prime $\equiv1\mod4$;\\
			$\mathbf{Output:}$ Four integers $a,b,c,d$ such that $a^2+b^2+c^2+d^2=w$;\\
			1. If $w$ is an odd prime $\equiv3\mod4$, use the method in \cite{Bum} to obtain a pair of integers $0<x,y<\frac{w}{2}$\\ such that $x^2+y^2\equiv-1\mod w$ in polynomial time.\\
			Otherwise, use the method in \cite{Rab} and derive $0\leq x,y<\frac{w}{2}$ such that $x^2+y^2\equiv-1\mod w$ in ran-\\dom polynomial time. \\
			2. Compute the Hurwitz expansion $\{a_k\}$ of $\frac{x+yi}{w}$ and $Q_k$ one by one until $|Q_{k+1}|^2-w\geq 0$;\\
			3. If $|Q_{k+1}|^2=w$, take $a=\Re(Q_{k+1})$, $b=\Im(Q_{k+1})$, $c=d=0$;\\
			4. If $|Q_{k+1}|^2>w$, take $a=\Re(Q_{k})$, $b=\Im(Q_{k})$, $c=\Re(Q_{k+1})$, $d=\Im(Q_{k+1})$; \\
			5. Return $a,b,c,d.$\\
			\hline
		\end{tabular}
	\end{center}
	
	\begin{prop}
		For a given odd $w$, steps $2$ to $5$ in Algorithm $1$ requires $O(\log w)$ operations. 
	\end{prop}
	It's quite straight to see that steps 2 to 5 in our algorithm is essentially the same as the method of calculating the greatest right common divisor of $w$ and $x+yi$ in Hurwitz order after finding satisfying $x,y$. Hence they share the same time complexity.\\
	We take $w=9878785333482266655552223331179$ as an example, which is a prime $\equiv3\mod4$.\\
	One can testify $(x,y)=$(3292928444494088885184074443726,2902967144089498477004731971911) is a pair of solution to the congruence equation $x^2+y^2\equiv-1\mod w$. Expanding $\frac{x+yi}{w}$ under Hurwitz algorithm we have the following results:
	\begin{center}
		\begin{tabular}{ ||c|c|c|c|| }
			\hline
			n & $a_n$ & $P_n$ & $Q_n$ \\ [0.5ex]
			\hline\hline
			-1 &  & 1 & 0 \\
			\hline
			0 & 0 & 0 & 1     \\
			\hline
			1 & 2 - $i$ & 1 & 2 - $i$  \\
			\hline
			2 & -1 + $i$ & -1 + $i$ & $3i$ \\
			\hline
			3 & -2$i$ & 3 + 2$i$ & 8 - $i$ \\
			\hline
			... & ... & ... & ... \\
			\hline
			36 & 3 + $i$ & 393331037760940 - 446167971615681$i$ & -1338503914847043$i$ \\
			\hline
			37 & -2$i$ & -805083291726049 - 974048629634780$i$ & -2808580912939087 - 446167971615681$i$ \\
			\hline
			38 & 2$i$ & 2341428297030500 - 2056334555067779$i$ & 892335943231362 - 6955665740725217$i$ \\
			\hline
		\end{tabular}
	\end{center}
	By calculation we have $|Q_{37}|<\sqrt{w}<|Q_{38}|$,  $|Q_{37}(x+yi)-P_{37}w|^2+|Q_{37}|^2=w$, thus we have obtained a representation of 9878785333482266655552223331179 as sum of squares:
	$$9878785333482266655552223331179=1338503914847043^2+2808580912939087^2+446167971615681^2$$
	Here we actually obtain a representation of $w$ as a sum of three squares due to the fact that $Q_{37}(x+yi)-P_{37}w$ is a pure imaginary number. From the Gauss-Legendre three-square theorem we know that the positive integers $n$ can be written as a sum of three integer squares if and only if $n$ is not in the form of $4^{k}(8 m+7)$ for any non-negative integers $k,m$. Hence the $w$ we choose does admit a three-square representation, and luckily we obtain one as above. If we choose another pair of solutions $x,y$, we may just obtain a four-square representation as usual. For example, if $(x,y)=(2469696333370566663888055832790, 386824569443398797078511524330)$, then we have
	$$w=807068241548931^2+2301810491935532^2+1089926588442075^2+1655643282228363^2$$
	
\end{document}